\documentclass[oneside,english]{amsart}
\usepackage[T1]{fontenc}
\usepackage[latin9]{inputenc}
\usepackage{amsthm}
\usepackage{amssymb}
\usepackage{graphicx}

\makeatletter
\numberwithin{equation}{section}
\numberwithin{figure}{section}
\theoremstyle{plain}
\newtheorem{thm}{\protect\theoremname}
  \theoremstyle{definition}
  \newtheorem{defn}[thm]{\protect\definitionname}
  \theoremstyle{remark}
  \newtheorem{rem}[thm]{\protect\remarkname}
  \theoremstyle{plain}
  \newtheorem{lem}[thm]{\protect\lemmaname}
  \theoremstyle{plain}
  \newtheorem{cor}[thm]{\protect\corollaryname}
  \theoremstyle{plain}
  \newtheorem{prop}[thm]{\protect\propositionname}

\makeatother

\usepackage{babel}
  \providecommand{\corollaryname}{Corollary}
  \providecommand{\definitionname}{Definition}
  \providecommand{\lemmaname}{Lemma}
  \providecommand{\propositionname}{Proposition}
  \providecommand{\remarkname}{Remark}
\providecommand{\theoremname}{Theorem}

\begin{document}

\title{Monotone Subsequences in High-Dimensional Permutations}

\author{Nathan Linial}

\address{School of Computer Science and Engineering, The Hebrew University
of Jerusalem, Jerusalem 91904, Israel.}

\email{nati@cs.huji.ac.il}

\thanks{Supported by ERC grant 339096 \textquotedbl{}High-Dimensional Combinatorics\textquotedbl{}.}

\author{Michael Simkin}

\address{Institute of Mathematics and Federmann Center for the Study of Rationality,
The Hebrew University of Jerusalem, Jerusalem 91904, Israel.}

\email{menahem.simkin@mail.huji.ac.il}
\begin{abstract}
This paper is part of the ongoing effort to study high-dimensional
permutations. We prove the analogue to the Erd\H{o}s-Szekeres theorem:
For every $k\ge1$, every order-$n$ $k$-dimensional permutation
contains a monotone subsequence of length $\Omega_{k}\left(\sqrt{n}\right)$,
and this is tight. On the other hand, and unlike the classical case,
the longest monotone subsequence in a random $k$-dimensional permutation
of order $n$ is asymptotically almost surely $\Theta_{k}\left(n^{\frac{k}{k+1}}\right)$.
\end{abstract}

\maketitle
The study of monotone subsequences in permutations began with the
famous Erd\H{o}s-Szekeres theorem \cite{erdos1935combinatorial}.
Since then numerous proofs and generalizations have emerged (see Steele's
survey \cite{steele1995variations}). We recall the theorem:
\begin{thm}
\label{thm:balanced est}Every permutation in $S_{n}$ contains a
monotone subsequence of length at least $\left\lceil \sqrt{n}\right\rceil $,
and this is tight: for every $n$ there exists some permutation in
$S_{n}$ in which all monotone subsequences are of length at most
$\left\lceil \sqrt{n}\right\rceil $.
\end{thm}
In order to derive a high-dimensional analogue of theorem \ref{thm:balanced est}
we need to define high-dimensional permutations and their monotone
subsequences. If we view a permutation as a sequence of distinct real
numbers, it is suggestive to consider sequences of points in $\mathbb{R}^{k}$,
with coordinatewise monotonicity. The following argument is attributed
by Kruskal \cite{kruskal1953monotonic} to de Bruijn: Repeatedly apply
theorem \ref{thm:balanced est} to conclude that every sequence $x_{1},x_{2},\ldots,x_{n}\in\mathbb{R}^{k}$
must have a coordinatewise monotone subsequence of length $n^{\frac{1}{2^{k}}}$,
and this is tight up to an additive constant. In \cite{kruskal1953monotonic}
one considers projections of the points to a line and defines the
length of the longest monotone subsequence according to the line with
the longest such subsequence. Szabó and Tardos \cite{szabo2001multidimensional}
consider sequences in $\mathbb{R}^{k}$ that avoid at least one of
the $2^{k}$ coordinatewise orderings.

Here we adopt the perspective of \cite{linial2014upper} of a high-dimensional
analogue of permutation matrices, and monotone subsequences are defined
by strict coordinatewise monotonicity. We show (theorem \ref{thm:balanced h.d. est})
that every $k$-dimensional permutation of order $n$ has a monotone
subsequence of length $\Omega_{k}\left(\sqrt{n}\right)$, and this
is tight up to the implicit multiplicative constant.

A related question, posed by Ulam \cite{ulam1961monte} in 1961, concerns
the distribution of $H_{n}^{1}$, the length of the longest increasing
subsequence in a random member of $S_{n}$. In 1972 Hammersley \cite{hammersley1972few}
showed that there exists some $C>0$ s.t.\ $H_{n}^{1}/\sqrt{n}$
converges to $C$ in probability. In 1977 Logan and Shepp \cite{logan1977variational}
showed that $C\geq2$ and Vershik and Kerov \cite{vershik1977asymptotics}
demonstrated that $C\leq2$, yielding the statement:
\begin{thm}
\label{thm:longest in random EST}Let $H_{n}^{1}$ be the length of
the longest increasing subsequence in a uniformly random member of
$S_{n}$. Then $\lim_{n\rightarrow\infty}H_{n}^{1}n^{-\frac{1}{2}}=2$
in probability.
\end{thm}
This result was famously refined in 1999 by Baik, Deift, and Johansson
\cite{baik1999distribution} who related the limiting distribution
of $H_{n}^{1}$ to the Tracy-Widom distribution.

Using coordinatewise monotonicity Bollobás and Winkler \cite{bollobas1988longest}
extended theorem \ref{thm:longest in random EST} to show that the
longest increasing subsequence among $n$ independently random points
in $\mathbb{R}^{k}$ is typically of length $c_{k}n^{\frac{1}{k}}$
for some $c_{k}\in\left(0,e\right)$. We show (theorem \ref{thm:random monotone})
that the longest monotone subsequence of a typical $k$-dimensional
permutation of order $n$ has length $\Theta_{k}\left(n^{\frac{k}{k+1}}\right)$.
A $k$-dimensional permutation can be viewed as a set of $n^{k}$
points in $\mathbb{R}^{k+1}$, and it is interesting to note this
asymptotic match with Bollobás and Winkler's result.

\section{Definitions and Main Results}

\textbf{Note:} throughout the paper all asymptotic expressions are
in terms of $n\to\infty$ and $k$ fixed.

As discussed in \cite{linial2014upper} and \cite{linial2014vertices},
we equate a permutation with the corresponding permutation matrix,
i.e., an $n\times n$ $\left(0,1\right)$-matrix in which each row
or column (henceforth, \emph{line}) contains a single $1$.\emph{
}We correspondingly define an \emph{order-$n$ $k$-dimensional permutation}
as an $\left[n\right]{}^{k+1}$ $\left(0,1\right)$-array in which
each line contains precisely one $1$. A \emph{line} in an $\left[n\right]{}^{k+1}$
array is comprised of all the positions obtained by fixing $k$ coordinates
and varying\ the remaining coordinate. We denote the set of order-$n$
$k$-dimensional permutations by $L_{n}^{k}$.

For a given $A\in L_{n}^{k}$ and $\alpha\in\left[n\right]^{k}$,
there is a unique $t\in\left[n\right]$ s.t.\ ${A\left(\alpha,t\right)=1}$.
Since $t$ is uniquely defined by $\alpha$, we can write $t=f_{A}(\alpha)$.
The function $f_{A}$ has the property that for every $1\leq j\leq k$
and $i_{1},\ldots,i_{j-1},i_{j+1},\ldots,i_{k}\in\left[n\right]$,
$\left\{ f_{A}\left(i_{1},\ldots,i_{j-1},t,i_{j+1},\ldots,i_{k}\right):1\leq t\leq n\right\} =\left[n\right]$.
In fact, the mapping $A\mapsto f_{A}$ is a bijection between $L_{n}^{k}$
and the family of $\left[n\right]{}^{k}$ arrays in which every line
contains each element in $\left[n\right]$. In dimension one this
is exactly the identification between permutation matrices and permutations.
This shows in particular that two-dimensional permutations, i.e.,
members of $L_{n}^{2}$, are order-$n$ \emph{Latin squares}.

We denote by $G_{A}$ the \emph{support} of $A\in L_{n}^{k}$, i.e.,
the set of $\alpha\in\left[n\right]^{k+1}$ s.t.\ $A\left(\alpha\right)=1$.

Higher-dimensional monotonicity is defined coordinatewise:
\begin{defn}
\label{def:monotonicity}A length-$m$ monotone subsequence in $A\in L_{n}^{k}$
is a sequence $\alpha^{1},\alpha^{2},\ldots,\alpha^{m}\in G_{A}$
s.t.\ for every $1\leq j\leq k+1$ the sequence $\alpha_{j}^{1},\alpha_{j}^{2},\ldots,\alpha_{j}^{m}$
is \emph{strictly} monotone.
\end{defn}
In dimension one this clearly coincides with the definition of a monotone
subsequence in a permutation $\pi\in S_{n}$.

We are now ready to state a high-dimensional analogue of the Erd\H{o}s-Szekeres
theorem:
\begin{thm}
\label{thm:balanced h.d. est}Every member of $L_{n}^{k}$ contains
a monotone subsequence of length $\Omega\left(\sqrt{n}\right)$. The
bound is tight up to the implicit multiplicative constant: for every
$n$ and $k$ there exists some $A\in L_{n}^{k}$ s.t.\ every monotone
subsequence in $A$ has length $O\left(\sqrt{n}\right)$.
\end{thm}
As an analogue of theorem \ref{thm:longest in random EST} we have:
\begin{thm}
Let $H_{n}^{k}$ be the length of the longest monotone subsequence
in a uniformly random element of $L_{n}^{k}$. Then $E\left[H_{n}^{k}\right]=\Theta\left(n^{\frac{k}{k+1}}\right)$
and $H_{n}^{k}=\Theta_{k}\left(n^{\frac{k}{k+1}}\right)$ a.a.s.\end{thm}
\begin{rem}
\label{rem:Latin symmetries}Note the following symmetries of high-dimensional
permutations: 
\begin{enumerate}
\item \label{enu:coordinate permutations}$S_{k+1}$ acts on $L_{n}^{k}$
by permuting the coordinates.
\item \label{enu:permutation of coordinate values}For each $1\leq i\leq k+1$,
the group $S_{n}$ acts on $L_{n}^{k}$ by permuting the values of
the $i$-th coordinate of each $A\in L_{n}^{k}$.
\item \label{enu:Reversal}A special case of \ref{enu:permutation of coordinate values},
is reversal, i.e. applying the map $a\mapsto1+n-a$ on the $i$-th
coordinate.
\end{enumerate}
Note that actions \ref{enu:coordinate permutations} and \ref{enu:Reversal}
preserve monotonicity.
\end{rem}

\section{A High-Dimensional Analogue of the Erd\H{o}s-Szekeres Theorem}

We begin by proving theorem \ref{thm:balanced h.d. est}. Due to the
Erd\H{o}s-Szekeres theorem it suffices to consider the case $k\geq2$.

We define two partial orders on $\left[n\right]^{k+1}$: Let $\alpha,\beta\in\left[n\right]^{k+1}$.
$\alpha<_{1}\beta$ if for all $1\leq i\leq k+1$, $\alpha_{i}<\beta_{i}$,
and $\alpha<_{2}\beta$ if for all $1\leq i\leq k$, $\alpha_{i}<\beta_{i}$
and, $\alpha_{k+1}>\beta_{k+1}$. For $\alpha,\beta\in\left[n\right]^{k}$
we write $\alpha<\beta$ if for all $1\leq i\leq k$, $\alpha_{i}<\beta_{i}$.

Recall that the \emph{height} $h\left(P\right)$ of a poset $P$ is
the size of the largest chain in $P$ and its \emph{width} $w\left(P\right)$
is the size of its largest anti-chain. An easy consequence of Dilworth's
theorem \cite{dilworth1950decomposition} or Mirsky's theorem \cite{mirsky1971dual}
is:
\begin{lem}
\label{lem:dilworth consequence}For every finite poset $P$ there
holds $h\left(P\right)\cdot w\left(P\right)\geq\left|P\right|$.
\end{lem}
We use lemma \ref{lem:dilworth consequence} to show that if $A$
has no long monotone subsequences, then there is a large $A'\subseteq G_{A}$
that is an anti-chain in both $<_{1}$ and $<_{2}$. On the other
hand, the next two lemmas give an upper bound on the size of anti-chains
common to $<_{1}$ and $<_{2}$. This yields the theorem.
\begin{lem}
\label{lem:box order}Let $X$ be an $M\times N$ matrix in which
every two entries in the same column are distinct. Let $S$ be a set
of positions in $X$ such that $X_{a}=X_{b}$ for every $a,b\in S$
with $a$ to the left and above $b$. Then $\left|S\right|\leq M+2N$.\end{lem}
\begin{proof}
If either $M=1$ or $N=1$, this is obvious. We prove the claim inductively
by showing that either $S$ has at most two positions in the rightmost
column of $X$ or at most one element in the topmost row of $X$.
Indeed, if $S$ has at least three entries in the rightmost column,
then at least two of them, say $a$ and $b$, are not in the top row.
But there are no repetitions in the same column, so $X_{a}\neq X_{b}$.
It follows that the only element $S$ may have in the top row is at
the top-right corner, for any other such element must equal both $X_{a}$
and $X_{b}$, which is impossible.
\end{proof}
For $k=2$ (i.e.\ the case of Latin squares) we already have the
necessary tools to prove the lower bound in theorem \ref{thm:balanced h.d. est}:
Let $A$ be an order-$n$ Latin square. A $<_{1}$ ($<_{2}$) monotone
sequence in $A$ is a sequence of positions progressing from upper-left
to lower-right in which the values of $f_{A}$ are increasing (decreasing).
Order $G_{A}$ by $<_{1}$, and assume there are no chains of length
$r=\left\lfloor \sqrt{\frac{n}{3}}\right\rfloor $. By lemma \ref{lem:dilworth consequence}
there is an $<_{1}$-anti-chain $A_{1}\subseteq G_{A}$ of size $\frac{\left|G_{A}\right|}{r}=\frac{n^{2}}{r}$.
Order $A_{1}$ by $<_{2}$ and let $A'\subseteq A_{1}$ be an anti-chain.
Note that $A'$ is an anti-chain under both $<_{1}$ and $<_{2}$.
If $S\subseteq\left[n\right]^{2}$ is the set of positions occupied
by the elements of $A'$, then $\left|S\right|\leq3n$ by lemma \ref{lem:box order}.
In other words $w(A_{1})\le3n$, so by lemma \ref{lem:dilworth consequence}
$h\left(A_{1}\right)\geq\frac{\left|A_{1}\right|}{w\left(A_{1}\right)}\geq\frac{n^{2}}{3nr}=\frac{n}{3r}\geq r$.
The height of $A_{1}$ is realized by a $<_{2}$-monotone subsequence
of length $h\left(A_{1}\right)\geq r=\left\lfloor \sqrt{\frac{n}{3}}\right\rfloor $,
which yields the lower bound.

We next extend lemma \ref{lem:box order} to higher dimensions, and
derive a similar technique to prove the general lower bound.
\begin{lem}
\label{lem:hypercube anti chains}Let $A\in L_{n}^{k}$, and let $S\subseteq\left[n\right]^{k}$
be s.t.\ if $a,b\in S$ and $a<b$, then $f_{A}\left(a\right)=f_{A}\left(b\right)$.
Then $\left|S\right|\leq3\left(k-1\right)n^{k-1}$.\end{lem}
\begin{proof}
\begin{figure}

\includegraphics[scale=0.4]{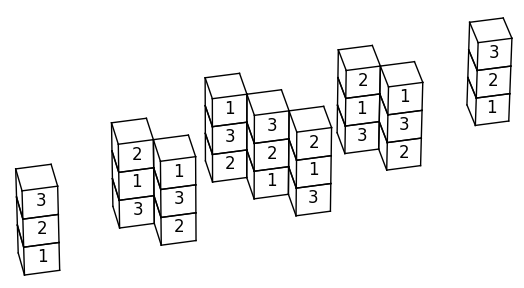}\caption{\label{fig:cube partition}Partitioning $\left[3\right]^{3}$ into
$P_{\alpha}$, $\alpha\in\left[3\right]^{2},\min\left\{ \alpha_{1},\alpha_{2}\right\} =1$.
The value of $f_{A}\left(x_{1},x_{2},x_{3}\right)=x_{1}+2x_{2}+x_{3}\left(\bmod3\right)$
is written on each position in the array.}

\end{figure}
The proof follows by partitioning $\left[n\right]^{k}$ and $S$ into
rectangular ``slices'' s.t.\ the restriction of $f_{A}$ to a single
slice satisfies the hypothesis of lemma \ref{lem:box order}: Let
${\mathbf{1}=\left(1,1,\ldots,1\right)\in\left[n\right]^{k-1}}$.
For $\alpha\in\left[n\right]^{k-1}$ we denote $M_{\alpha}=n-\max\left\{ \alpha_{1},\alpha_{2},\ldots,\alpha_{k-1}\right\} $.
For every $\alpha\in\left[n\right]^{k-1}$ s.t.\ $\min\left\{ \alpha_{1},\alpha_{2},\ldots,\alpha_{k-1}\right\} =1$,
let 
\[
P_{\alpha}=\left\{ \left(\alpha+m\mathbf{1},i\right):0\le m\le M_{\alpha},1\leq i\leq n\right\} \subseteq\left[n\right]^{k}
\]
 and $S_{\alpha}=P_{\alpha}\cap S$. It is easily verified that the
distinct $P_{\alpha}$ constitute a partition of $\left[n\right]^{k}$
and hence the $S_{\alpha}$ are a partition of $S$ (see figure \ref{fig:cube partition}).

We associate with $P_{\alpha}$ the $M\times n$ matrix $X$ given
by $X_{a,b}=f_{A}\left(\alpha+a\mathbf{1},b\right)$. Under this identification
the partial order $<$ on $P_{\alpha}$ corresponds to being above
and to the left in $X$: $\left(\alpha+a_{1}\mathbf{1},b_{1}\right)<\left(\alpha+a_{2}\mathbf{1},b_{2}\right)\iff a_{1}<a_{2}\land b_{1}<b_{2}$.
Since $A$ is a $k$-dimensional permutation, no two entries in the
same column of $X$ are equal. Apply lemma \ref{lem:box order} to
$S_{\alpha}$, to conclude that $\left|S_{\alpha}\right|\leq2M+n\leq3n$.

There are fewer than $\left(k-1\right)n^{k-2}$ multi-indices $\alpha\in\left[n\right]^{k-1}$
s.t.\ ${\min\left\{ \alpha_{1},\alpha_{2},\ldots,\alpha_{k-1}\right\} =1}$,
and so $\left|S\right|\leq3\left(k-1\right)n^{k-1}$, as desired.
\end{proof}
We are now ready to prove theorem \ref{thm:balanced h.d. est}:
\begin{proof}
Let $A\in L_{n}^{k}$ and assume $A$ has no $<_{1}$-chain of length
$r=\left\lfloor \sqrt{\frac{n}{3\left(k-1\right)}}\right\rfloor $.
By lemma \ref{lem:dilworth consequence} $G_{A}$ has a $<_{1}$-anti-chain
$A_{1}$ of cardinality $\left|A_{1}\right|\geq\frac{n^{k}}{r}$.
Order $A_{1}$ according to $<_{2}$ and let $A_{2}\subseteq A_{1}$
be a $<_{2}$-anti-chain, so that $A_{2}$ is an anti-chain under
both $<_{1}$ and $<_{2}$. Let $S=\left\{ x\in\left[n\right]^{k}:\left(x,f_{A}\left(x\right)\right)\in A_{2}\right\} $
be the projection of $A_{2}$ onto the first $k$ coordinates. Since
$f_{A}$ is a function, the projection is injective and so $\left|S\right|=\left|A_{2}\right|$.
By lemma \ref{lem:hypercube anti chains}, $\left|A_{2}\right|=\left|S\right|\leq3\left(k-1\right)n^{k-1}$.
Thus, $w(A_{1})\le3\left(k-1\right)n^{k-1}$, and by lemma \ref{lem:dilworth consequence}
$h(A_{1})\ge\frac{\left|A_{1}\right|}{3\left(k-1\right)n^{k-1}}\geq\frac{n^{k}}{3\left(k-1\right)n^{k-1}r}\geq r$.
But $A_{1}$'s height is realized by a $<_{2}$-chain in $A_{1}\subseteq G_{A}$,
and so $G_{A}$ contains an $<_{2}$-chain of length at least $r=\left\lfloor \sqrt{\frac{n}{3\left(k-1\right)}}\right\rfloor $.
This completes the proof of the first part of the theorem.

For the second part of the theorem, for every $n$ and $k$ we construct
$A\in L_{n}^{k}$ with no monotone subsequences of length $O\left(\sqrt{n}\right)$.
When $n$ is prime, we can use a simple construction that is similar
to one that shows the tightness of the Erd\H{o}s-Szekeres theorem.
This construction is then modified to deal with composite $n$. So
let us assume that $n$ is prime. Let $M=\left\lfloor \sqrt{\frac{n}{k+1}}\right\rfloor $,
and define $A$ as follows:
\[
A\left(\alpha_{1},\alpha_{2},\ldots,\alpha_{k+1}\right)=1\iff M\sum_{i=1}^{k}\alpha_{i}+\alpha_{k+1}=0\left(\bmod n\right)
\]

Since $n$ is prime it follows easily that $A$ is a $k$-dimensional
permutation.

We want to show that if $\alpha^{1},\alpha^{2},\ldots,\alpha^{m}\in G_{A}$
is a monotone subsequence, then for every $1\leq j<m$, $\left\Vert \alpha^{j+1}-\alpha^{j}\right\Vert _{1}$
is large. Because the sequence is monotone we have ${\sum_{j=1}^{m-1}\left\Vert \alpha^{j+1}-\alpha^{j}\right\Vert _{1}=\left\Vert \alpha^{m}-\alpha^{1}\right\Vert _{1}\leq\left(k+1\right)n}$,
which gives an upper bound on $m$.

We may assume w.l.o.g.\ that the $\alpha_{k+1}^{j}$s are increasing;
otherwise, consider the monotone subsequence $\alpha^{m},\alpha^{m-1},\ldots,\alpha^{1}$.
We partition the set of coordinates ${\left[k\right]=S^{+}\dot{\cup}S^{-}}$
into those on which the sequence is increasing, resp. decreasing.
Since all the $\alpha^{j}$s satisfy the same linear equation modulo
$n$, for every $1\leq j<m$ there is some $d_{j}\in\mathbb{Z}$ s.t.
\begin{equation}
M\sum_{i\in S^{+}}\left|\alpha_{i}^{j+1}-\alpha_{i}^{j}\right|+\alpha_{k+1}^{j+1}-\alpha_{k+1}^{j}=M\sum_{i\in S^{-}}\left|\alpha_{i}^{j+1}-\alpha_{i}^{j}\right|+d_{j}n\label{eq:linear relation}
\end{equation}

Let $\ell_{-},\ell_{0},\ell_{+}$ be the number of indices $j$ s.t.
$d_{j}<0,d_{j}=0,d_{j}>0$, respectively. Then $m=\ell_{-}+\ell_{0}+\ell_{+}+1$.
If $d_{j}<0$ then $\sum_{i\in S^{-}}\left|\alpha_{i}^{j+1}-\alpha_{i}^{j}\right|>\frac{n}{M}$.
Hence 
\[
\ell_{-}\frac{n}{M}<\sum_{j=1}^{m-1}\sum_{i\in S^{-}}\left|\alpha_{i}^{j+1}-\alpha_{i}^{j}\right|=\sum_{i\in S^{-}}\left|\alpha_{i}^{m}-\alpha_{i}^{1}\right|<\left|S^{-}\right|n
\]
 so we have $\ell_{-}<M\left|S^{-}\right|$. Similarly, if $d_{j}>0$
then
\[
M\sum_{i\in S^{+}\cup\left\{ k+1\right\} }\left|\alpha_{i}^{j+1}-\alpha_{i}^{j}\right|\geq M\sum_{i\in S^{+}}\left|\alpha_{i}^{j+1}-\alpha_{i}^{j}\right|+\alpha_{k+1}^{j+1}-\alpha_{k+1}^{j}>n
\]
and $\ell_{+}<M\left(\left|S^{+}\right|+1\right)$. If $d_{j}=0$,
taking equation \ref{eq:linear relation} modulo $M$ shows that $\alpha_{k+1}^{j+1}-\alpha_{k+1}^{j}=0\left(\bmod M\right)$.
Since $\alpha_{k+1}^{j+1}-\alpha_{k+1}^{j}\neq0$ we have $\alpha_{k+1}^{j+1}-\alpha_{k+1}^{j}\geq M$.
Therefore ${\ell_{0}M\leq\alpha_{k+1}^{m}-\alpha_{k+1}^{1}<n}$ so
$\ell_{0}<\frac{n}{M}$. Putting everything together, we have: 
\[
m=\ell_{-}+\ell_{0}+\ell_{+}+1<M\left(k+1\right)+\frac{n}{M}+1\leq2\sqrt{n\left(k+1\right)}+1+o\left(1\right)
\]
 yielding the upper bound.

In this construction we need $M$ and $n$ to be relatively prime.
For composite $n$ this isn't necessarily the case, and we offer two
remedies: The first is an appeal to number theory to produce $M\approx\sqrt{\frac{n}{k+1}}$
coprime to $n$. It is known \cite{baker2001difference} that for
large $x$, there is always a prime in the interval $\left[x-x^{0.525},x\right]$.
Therefore, we can find three distinct primes in an interval $\left[\sqrt{\frac{n}{k+1}},\left(1+o\left(1\right)\right)\sqrt{\frac{n}{k+1}}\right]$.
At least one of these must be coprime to $n$, since their product
exceeds $n$ for large $n$. This implies that all monotone subsequences
have length $\le\left(2+o\left(1\right)\right)\sqrt{\left(k+1\right)n}$.

The second approach is easy to generalize, as done in the proof of
theorem \ref{thm:erdos-szekeres-high-dim}. Take $M=\left\lfloor \sqrt{\frac{n}{k+1}}\right\rfloor $
as before. Let $g=\gcd\left(M,n\right)$ and define the permutation
$\pi\in S_{n}$ as follows (all values are taken modulo $n$):
\[
\pi=\left(M,2M,\ldots,\frac{n}{g}M,1+M,\ldots,1+\frac{n}{g}M,\ldots,g-1+M,\ldots,g-1+\frac{n}{g}M\right)
\]
Set $f_{A}\left(\alpha_{1},\alpha_{2},\ldots,\alpha_{k}\right)=-\pi\left(\sum_{i=1}^{k}\alpha_{i}\right)$.
Note that if $\gcd\left(M,n\right)=1$, this coincides with the construction
above. Now, if $\alpha^{1},\alpha^{2},\ldots,\alpha^{m}\in G_{A}$
is a monotone subsequence increasing in the last coordinate we bound
it in a similar manner to the calculation above. Using the same notations
we have, for every $1\leq j<m$:
\[
M\sum_{i\in S^{+}}\left|\alpha_{i}^{j+1}-\alpha_{i}^{j}\right|+\alpha_{k+1}^{j+1}-\alpha_{k+1}^{j}=M\sum_{i\in S^{-}}\left|\alpha_{i}^{j+1}-\alpha_{i}^{j}\right|+d_{j}n+r_{j}
\]
where $\left|r_{j}\right|<g\leq M$. If $d_{j}<0$ we have $M\sum_{i\in S^{-}}\left|\alpha_{i}^{j+1}-\alpha_{i}^{j}\right|>n-M$,
so ${\ell_{-}<\frac{n}{n-M}\left|S^{-}\right|}$. Similarly ${\ell_{+}<\frac{n}{n-M}\left(\left|S^{+}\right|+1\right)}$.
If $d_{j}=0$ and $\alpha_{k+1}^{j+1}-\alpha_{k+1}^{j}<M$, then we
have $r_{j}\neq0$. But then, by the definition of $\pi$, we must
have 
\[
{\sum_{i=0}^{k}\left(\alpha_{i}^{j+1}-\alpha_{i}^{j}\right)\geq\frac{n}{g}-1\geq\frac{n}{M}-1}
\]
Therefore $\ell_{0}<\frac{kn}{\frac{n}{M}-1}+\frac{n}{M}=\left(2+o\left(1\right)\right)\sqrt{\left(k+1\right)n}$.
So 
\[
{m<\ell_{-}+\ell_{0}+\ell_{+}+1\leq\left(3+o\left(1\right)\right)\sqrt{\left(k+1\right)n}}
\]

\end{proof}
Most proofs of theorem \ref{thm:balanced est} actually yield a more
general statement:
\begin{thm}
\label{thm:erdos szkeres}Let $r,s$ and $n$ be positive integers
with $rs<n$. Then every permutation in $S_{n}$ contains either an
increasing subsequence of length $r+1$, or a decreasing subsequence
of length $s+1$. The bound is tight: if $rs\ge n$ then there is
a permutation in $S_{n}$ with neither an increasing subsequence of
length $r+1$ nor a decreasing subsequence of length $s+1$.
\end{thm}
It is possible to extend theorem \ref{thm:balanced h.d. est} in a
similar fashion. To this end we refine our notion of monotonicity.
In dimension one we distinguish between ascending and descending subsequences,
and we need something similar in higher dimensions.
\begin{defn}
\label{def:monotone seq}A vector $\vec{c}\in\left\{ 0,1\right\} ^{k+1}$
induces a partial order $x<_{\vec{c}}y$ on $\mathbb{R}^{k+1}$ as
follows: For $1\leq i\leq k+1$ if $c_{i}=1$, then $x_{i}<y_{i}$
and if $c_{i}=0$, then $y_{i}<x_{i}$.\end{defn}
\begin{thm}
\label{thm:erdos-szekeres-high-dim}Let $\vec{c},\vec{d}\in\left\{ 0,1\right\} ^{k+1}$
differ in exactly one coordinate. Let $rs<\frac{n}{3(k-1)}$. Then
every $A\in L_{n}^{k}$, contains either a $<_{\vec{c}}$-monotone
subsequence of length $r$ or a $<_{\vec{d}}$-monotone subsequence
of length $s$.

The bound is tight up to the multiplicative constants: If $r,s\geq9\left(k+10\right)$
and $rs>5kn$, then there exists $A\in L_{n}^{k}$ with no $<_{\vec{c}}$-monotone
subsequence of length $r$ nor a $<_{\vec{d}}$-monotone subsequence
of length $s$.\end{thm}
\begin{proof}
Using the symmetries from remark \ref{rem:Latin symmetries} we may
assume w.l.o.g.\ that $<_{\vec{c}}=<_{1}$ and $<_{\vec{d}}=<_{2}$. 

The proof of the lower bound is similar to the proof of the lower
bound in theorem \ref{thm:balanced h.d. est}, and we provide only
a sketch. Lemma \ref{lem:hypercube anti chains} gives an upper bound
of $3\left(k-1\right)n^{k-1}$ on the size of any anti-chain under
both $<_{1}$ and $<_{2}$. Two applications of lemma \ref{lem:dilworth consequence}
yield the lower bound.

For the upper bound, assume w.l.o.g.\ that $r\geq s$. We construct
$\pi\in S_{n}$ and $A\in L_{n}^{k}$ as before, with $M=\left\lfloor \frac{s}{2k}\right\rfloor $.
Let $\alpha^{1},\alpha^{2},\ldots,\alpha^{m}\in G_{A}$ be a $<_{1}$-monotone
subsequence. Then the sequence is increasing in every coordinate.
For all $j$, if $\alpha_{k+1}^{j+1}-\alpha_{k}^{j}<M$ then $\sum_{i=1}^{k}\left(\alpha_{i}^{j+1}-\alpha_{i}^{j}\right)\geq\frac{n}{g}\geq\frac{n}{M}$.
Thus 
\[
m\leq\frac{n}{M}+\frac{kn}{\frac{n}{M}}+1=\frac{n}{M}+kM+1\leq\frac{2kn}{s}\left(1+\frac{2k}{s}\right)+\frac{s}{2}+1
\]
Using the assumptions that $\frac{r}{5k}>\frac{n}{s}$ and $r\geq s\geq9\left(k+10\right)$,
we have:
\[
m\leq r\left(\frac{2}{5}\left(1+\frac{2}{9}\right)+\frac{1}{2}+\frac{1}{r}\right)\leq r
\]

Now, let $\alpha^{1},\alpha^{2},\ldots,\alpha^{m}\in G_{A}$ be a
$<_{2}$-monotone subsequence. For $1\leq j\leq m$ define $s_{j}=M\sum_{i=1}^{k}\alpha_{i}^{j}$.
This is an increasing sequence, and $s_{j+1}-s_{j}\geq M$ for all
$j$. By definition of $A$, $\alpha_{k+1}^{j}=s_{j}\left(\bmod n\right)+r_{j}$
for some $0\leq r_{j}<M$. Because $\alpha_{k+1}^{1},\alpha_{k+1}^{2},\ldots,\alpha_{k+1}^{m}$
is decreasing, if for some $j$, $s_{j}$ and $s_{j+1}$ fall in the
same interval of the form $\left[dn+1,\left(d+1\right)n\right]$ (for
$d\in\mathbb{Z}$), then $s_{j}+r_{j}>s_{j+1}\implies s_{j+1}-s_{j}<r_{j}<M$,
a contradiction. Therefore the $s_{j}$'s fall into distinct intervals
of the form $\left[dn+1,\left(d+1\right)n\right]$. But for every
$j$, $0<s_{j}\leq Mkn$. Since $\left[0,Mkn\right]$ contains only
$\left\lceil \frac{Mkn}{n}\right\rceil \leq Mk+1$ intervals of length
$n$, we have $m\leq Mk+1\leq\frac{s}{2}+1<s$.
\end{proof}

\section{\label{sec:random}Monotone Subsequences in Random High-Dimensional
Permutations}

As mentioned in the introduction, the longest monotone subsequence
of a random permutation is typically of length $2\sqrt{n}$. In view
of the Erd\H{o}s-Szekres theorem this means that the random case and
the worst case are of the same order of magnitude and differ by only
a constant factor. In higher dimensions this is no longer the case.
The longest monotone subsequence of a typical element in $L_{n}^{k}$
has length $\Theta_{k}\left(n^{\frac{k}{k+1}}\right)$.

We define the random variable $H_{n}^{k}$ - the length of the longest
monotone subsequence in a uniformly random element of $L_{n}^{k}$,
and prove:
\begin{thm}
\label{thm:random monotone}For every $k\in\mathbb{N}$:
\begin{enumerate}
\item For every $\varepsilon>0$, ${H_{n}^{k}n^{-\frac{k}{k+1}}\in\left[\frac{1}{k+2},e+\varepsilon\right]}$
asymptotically almost surely.
\item $1-\frac{\ln k+1}{k+1}-o\left(1\right)\leq E\left[H_{n}^{k}n^{-\frac{k}{k+1}}\right]\leq e+o\left(1\right)$.
\end{enumerate}
\end{thm}
There are $2^{k+1}$ distinct order types of monotone subsequences,
indexed by binary vectors $\vec{c}\in\left\{ 0,1\right\} ^{k+1}$.
By permuting coordinates we see that the distribution of the longest
$<_{\vec{c}}$-monotone subsequence in a random element of $L_{n}^{k}$
is independent of $\vec{c}$. Thus it suffices to prove theorem \ref{thm:random monotone}
for $<_{\left(1,1,\ldots,1\right)}$-monotone subsequences. For brevity
of notation we write $<$ in place of $<_{\left(1,1,\ldots,1\right)}$.

The following lemma is useful in dealing with uniformly random elements
of $L_{n}^{k}$:
\begin{lem}
\label{lem:square and perm}Given $A\in L_{n}^{k}$ and $\pi\in S_{n}$,
let $\pi\left(A\right)\in L_{n}^{k}$ be obtained by permuting the
first coordinate of $G_{A}$ according to $\pi$. If $A$ and $\pi$
are chosen independently uniformly at random, then $\pi\left(A\right)$
is uniformly distributed in $L_{n}^{k}$.\end{lem}
\begin{proof}
It's enough to show that the $S_{n}$-action on $L_{n}^{k}$ described
in the statement is free. Indeed, assume $\pi\left(A\right)=A$. Let
$i\in\left[n\right]$. There exists a unique $x$ s.t.\ $A_{\left(i,x,1,\ldots,1\right)}=1$.
Since $\pi\left(A\right)=A$ we have $A_{\left(\pi\left(i\right),x,1,\ldots,1\right)}=1$.
But then $\pi\left(i\right)=i$ is the unique $y$ s.t.\ $A_{\left(y,x,1,\ldots,1\right)}=1$.
Thus $\pi$ is the identity.
\end{proof}
We have arbitrarily chosen for $\pi$ to act on the first coordinate,
but occasionally (e.g., in the next lemma) we have it act on other
coordinates, as needed.

A useful corollary of this lemma follows:
\begin{cor}
\label{lem:prob increasing}Let $\alpha_{1},\alpha_{2},\ldots,\alpha_{m}\in\left[n\right]^{k}$
be distinct positions. For a uniformly drawn $A\in L_{n}^{k}$, 
\[
\Pr\left[f_{A}\left(\alpha_{1}\right)<f_{A}\left(\alpha_{2}\right)<\ldots<f_{A}\left(\alpha_{m}\right)\right]\le\frac{1}{m!}
\]

\end{cor}
We first prove the upper bounds in theorem \ref{thm:random monotone}.
\begin{prop}
~
\begin{enumerate}
\item For every $\varepsilon>0$ there holds $\Pr\left[H_{n}^{k}n^{-\frac{k}{k+1}}>e+\varepsilon\right]=o(1)$.
\item $E\left[H_{n}^{k}\right]n^{-\frac{k}{k+1}}\le e+o\left(1\right)$.
\end{enumerate}
\end{prop}
\begin{proof}
We bound the expected number of length-$m$ monotone subsequences
in a random $k$-dimensional permutation. For every increasing sequence
of positions $\alpha=\alpha_{1},\alpha_{2},\ldots,\alpha_{m}\in\left[n\right]^{k}$
and $A\in L_{n}^{k}$ we define 
\[
X_{\alpha}\left(A\right)=\begin{cases}
1 & f_{A}\left(\alpha_{1}\right)<f_{A}\left(\alpha_{2}\right)<\ldots<f_{A}\left(\alpha_{m}\right)\\
0 & otherwise
\end{cases}
\]
By lemma \ref{lem:prob increasing} $E\left[X_{\alpha}\left(A\right)\right]=\Pr\left[X_{\alpha}\left(A\right)=1\right]\leq\frac{1}{m!}$
for a uniform $A\in L_{n}^{k}$. Let $S$ be the set of all length-$m$
increasing sequences of positions in $\left[n\right]^{k}$. Clearly,
$|S|=\left(\begin{array}{c}
n\\
m
\end{array}\right)^{k}$so by linearity of expectation:
\[
\Pr\left[H_{n}^{k}\geq m\right]=\Pr\left[\sum_{\alpha\in S}X_{\alpha}\left(A\right)>0\right]\leq E\left[\sum_{\alpha\in S}X_{\alpha}\left(A\right)\right]\leq\frac{\left(\begin{array}{c}
n\\
m
\end{array}\right)^{k}}{m!}\leq\left(\left(\frac{e}{m}\right)^{k+1}n^{k}\right)^{m}
\]

Let $c=e+\varepsilon$ for some $\varepsilon>0$, and let $m=cn^{\frac{k}{k+1}}$.
Then: 
\[
\Pr\left[H_{n}^{k}n^{-\frac{k}{k+1}}>c\right]=\Pr\left[H_{n}^{k}\geq m\right]\leq\left(\frac{e}{c}\right)^{\left(k+1\right)cn^{\frac{k}{k+1}}}=o(1)
\]
proving the first claim in the proposition. Further:
\[
E\left[H_{n}^{k}\right]n^{-\frac{k}{k+1}}\leq\left(m\Pr\left[H_{n}^{k}<m\right]+n\Pr\left[H_{n}^{k}\geq m\right]\right)n^{-\frac{k}{k+1}}\leq c+n^{\frac{1}{k+1}}\left(\frac{e}{c}\right)^{\left(k+1\right)cn^{\frac{k}{k+1}}}=c+o(1)
\]
which proves the second claim.
\end{proof}
The proof of the lower bounds is more intricate. Fix some $C>0$ and
let $m=Cn^{\frac{1}{k+1}}$. For $1\leq i\leq\left\lfloor \frac{n}{m}\right\rfloor $,
let $D_{i}=\left[\left(i-1\right)m+1,im\right]^{k+1}$ be the diagonal
subcubes of $\left[n\right]^{k+1}$. For a uniformly random $A\in L_{n}^{k}$
let $Z_{i}$ be the indicator random variable of the event that $A$
is not all zero on $D_{i}$. Clearly, $H_{n}^{k}\geq\sum_{1\leq i\leq\frac{n}{m}}Z_{i}$,
since $\alpha<\beta$ if $\alpha\in D_{i},\beta\in D_{j}$, and $i<j$.
Indeed we prove lower bounds on $H_{n}^{k}$ by bounding $\sum_{1\leq i\leq\frac{n}{m}}Z_{i}$.
It is convenient to express everything in terms of the random variable
$Y_{n}=n^{-\frac{k}{k+1}}\sum_{1\leq i\leq\frac{n}{m}}Z_{i}$. We
show that for an appropriate choice of $C$ (see below) $Y_{n}$ converges
in probability to a constant in $(0,1)$. These are our main steps:
\begin{enumerate}
\item \label{one}Note that $Y_{n}\le\frac{1}{C}+o\left(1\right)$.
\item \label{two}Prove that $E\left[Y_{n}\right]\ge\frac{C^{k}}{C^{k+1}+1}-o\left(1\right)$
(proposition \ref{prop:lower bound expectation}).
\item \label{three}Show that if $C<1$, then $\Pr\left[Y_{n}>C^{k+1}+\varepsilon\right]=o\left(1\right)$
for every $\varepsilon>0$ (corollary \ref{cor:chernoff}).
\item \label{four}By letting $1>C>0$ be the unique solution to $\frac{C^{k}}{1+C^{k+1}}=C^{k+1}$,
conclude that $\Pr\left[Y_{n}<C^{k+1}-\varepsilon\right]=o\left(1\right)$
for every $\varepsilon>0$ (proposition \ref{prop:.lower tail bound}).
Hence $\lim_{n\rightarrow\infty}Y_{n}=C^{k+1}$ in probability.
\end{enumerate}
In step \ref{one} we assume only that $C>0$. The claim in step \ref{two}
applies to all $C>0$, and we optimize the bound on $E\left[Y_{n}\right]$
by a particular choice of $C$. Step \ref{three} applies to all $1>C>0$.
Finally in step \ref{four} we assign a value to $C$ to derive the
conclusion that $Y_{n}$ converges in probability to $C^{k+1}$.

We start with step \ref{two}, a lower bound on $E\left[Y_{n}\right]$:
\begin{lem}
\label{lem:home cube prob}For $1\leq i\le\frac{n}{m}$, $\Pr\left[Z_{i}=1\right]\geq\frac{C^{k+1}}{C^{k+1}+1}-o\left(1\right)$.\end{lem}
\begin{proof}
Let $X_{i}=\sum_{\alpha\in D_{i}}A_{\alpha}$ be the number of non-zero
entries in $D_{i}$. Note that $X_{i}>0\iff Z_{i}=1$. We prove a
lower bound on the probability of this event by a second moment argument.

Clearly, $E\left[X_{i}\right]=\frac{\left|D_{i}\right|}{n}=C^{k+1}$,
since $\Pr\left[A_{\alpha}=1\right]=\frac{1}{n}$ for every $\alpha\in\left[n\right]^{k+1}$. 

We next seek an upper bound on $E\left[X_{i}^{2}\right]$. 
\[
E\left[X_{i}^{2}\right]=\sum_{\alpha,\beta\in D_{i}}E\left[A_{\alpha}A_{\beta}\right]=\sum_{\alpha,\beta\in D_{i}}\Pr\left[A_{\alpha}A_{\beta}=1\right]
\]
There are $m^{k+1}$ terms with $\alpha=\beta$, each being $\frac{1}{n}$.

To consider $\alpha\neq\beta$, assume w.l.o.g.\ that ${\alpha=\left(1,\alpha_{2},\alpha_{3},\ldots,\alpha_{k+1}\right)},{\beta=\left(2,\beta_{2},\beta_{3},\ldots,\beta_{k+1}\right)}$.
There exist unique $x,y\in\left[n\right]$ s.t.\ $A_{\left(x,\alpha_{2},\alpha_{3},\ldots,\alpha_{k+1}\right)}=A_{\left(y,\beta_{2},\beta_{3},\ldots,\beta_{k+1}\right)}=1$.
Choose a random $\pi\in S_{n}$ and $\pi\left(A\right)$ denote the
$k$-dimensional permutation obtained by permuting the first coordinate
of $G_{A}$ according to $\pi$. The event $\pi\left(A\right)_{\alpha}\pi\left(A\right)_{\beta}=1$
is the same as $\pi\left(x\right)=1$ and $\pi\left(y\right)=2$ and
its probability is $\frac{1}{n\left(n-1\right)}$ if $x\neq y$, and
$0$ otherwise. In general, $\alpha$ and $\beta$ may agree on the
first coordinate and differ elsewhere. In this case choose some coordinate
on which they disagree and permute it according to $\pi$, to obtain
the same bound on the probability.  There are fewer than $m^{2\left(k+1\right)}$
such pairs $\alpha,\beta\in D_{i}$, so

\[
E\left[X_{i}^{2}\right]=\sum_{\alpha,\beta\in D_{i}}\Pr\left[A_{\alpha}A_{\beta}=1\right]\leq m^{k+1}\left(\frac{1}{n}+\frac{m^{k+1}}{n\left(n-1\right)}\right)=\frac{m^{k+1}}{n}\left(1+\frac{m^{k+1}}{n-1}\right)
\]
Noting that $E\left[X_{i}\right]=\frac{m^{k+1}}{n}=C^{k+1}$, we have:
\[
E\left[X_{i}^{2}\right]\leq E\left[X_{i}\right]\left(1+\frac{n}{n-1}E\left[X_{i}\right]\right)=C^{k+1}\left(1+\frac{n}{n-1}C^{k+1}\right)
\]
The second moment method yields:
\[
\Pr\left[Z_{i}=1\right]=\Pr\left[X_{i}>0\right]\geq\frac{E\left[X_{i}\right]^{2}}{E\left[X_{i}^{2}\right]}=\frac{C^{k+1}}{C^{k+1}+1+o\left(1\right)}=\frac{C^{k+1}}{C^{k+1}+1}-o\left(1\right)
\]

\end{proof}
We conclude:
\begin{prop}
\label{prop:lower bound expectation}~

$E\left[Y_{n}\right]\geq\frac{C^{k}}{C^{k+1}+1}-o\left(1\right)$,
consequently $E\left[n^{-\frac{k}{k+1}}H_{n}^{k}\right]\geq1-\frac{\ln k+1}{k+1}-o(1)$.\end{prop}
\begin{proof}
As observed earlier:
\[
E\left[Y_{n}\right]=E\left[n^{-\frac{k}{k+1}}\sum_{1\leq i\leq\frac{n}{m}}Z_{i}\right]=n^{-\frac{k}{k+1}}\left\lfloor \frac{n}{m}\right\rfloor \Pr\left[Z_{i}=1\right]
\]
So, by lemma \ref{lem:home cube prob}:
\[
E\left[Y_{n}\right]\geq\frac{C^{k}}{C^{k+1}+1}-o\left(1\right)
\]
For all $C$, $E\left[n^{-\frac{k}{k+1}}H_{n}^{k}\right]\geq E\left[Y_{n}\right]$.
The optimal bound is attained when $C=k^{\frac{1}{k+1}}$, yielding:
\[
E\left[n^{-\frac{k}{k+1}}H_{n}^{k}\right]\geq\frac{k^{\frac{k}{k+1}}}{k+1}-o\left(1\right)\geq1-\frac{\ln k+1}{k+1}-o\left(1\right)
\]

\end{proof}
To prove the lower bound in theorem \ref{thm:random monotone} part
(1), we apply a Chernoff bound to the events $\left\{ Z_{i}=1\right\} _{1\leq i\leq\frac{n}{m}}$.
To overcome the dependencies among these events we utilize the version
of the Chernoff inequality from (\cite{impagliazzo2010constructive},
theorem 1.1):
\begin{thm}
\label{thm:generalized chernoff}Let $0\le\alpha\leq\beta\leq1$ and
let $\left\{ X_{i}\right\} _{i\in\left[N\right]}$ be Boolean random
variables such that for all $S\subseteq\left[N\right]$, $\Pr\left[\prod_{i\in X}X_{i}=1\right]\leq\alpha^{\left|S\right|}$.
Then\\
 ${\Pr\left[\sum_{i\in\left[N\right]}X_{i}\geq\beta N\right]\leq e^{-ND\left(\beta\parallel\alpha\right)}}$,
where $D\left(\beta\parallel\alpha\right)=\beta\ln\left(\frac{\beta}{\alpha}\right)+\left(1-\beta\right)\ln\left(\frac{1-\beta}{1-\alpha}\right)$
is the relative entropy function.\end{thm}
\begin{lem}
\label{lem:multiplicative prob}Assume $C<1$. Let $S\subseteq\left\{ 1,2,\ldots,\left\lfloor \frac{n}{m}\right\rfloor \right\} $.
Then ${\Pr\left[\prod_{i\in S}Z_{i}=1\right]\leq\alpha^{\left|S\right|}}$
for all $C^{k+1}<\alpha<1$ and large enough $n$.\end{lem}
\begin{proof}
Note that $Z_{i}=1$ for all $i\in S$ iff there exist positions $\left\{ \beta^{i}\right\} _{i\in S}$
s.t.\ $\beta^{i}\in D_{i}$ for all $i\in S$ and $A_{\beta^{i}}=1$
for all $i$. We bound the probability of this occurrence using a
union bound.

Let $\left\{ \beta^{i}\right\} _{i\in S}$ be positions s.t.\ $\beta^{i}\in D_{i}$
for all $i\in S$. Generate a uniformly random $\pi(A)\in L_{n}^{k}$
by uniformly and independently drawing $A\in L_{n}^{k}$ and $\pi\in S_{n}$,
and applying $\pi$ to the first coordinate of $G_{A}$. For every
$i$, there exists a unique $x_{i}\in\left[n\right]$ s.t.\ $A_{\left(x_{i},\beta_{2}^{i},\ldots,\beta_{k+1}^{i}\right)}=1$.
So $\pi\left(A\right)_{\beta^{i}}=1$ iff $\pi\left(x_{i}\right)=\beta_{1}^{i}$.
Thus $\pi\left(A\right)_{\beta^{i}}=1$ for all $i\in S$ iff $\pi\left(x_{i}\right)=\beta_{1}^{i}$
for all $i\in S$, and this occurs with probability at most $\frac{\left(n-\left|S\right|\right)!}{n!}$.
There are $m^{\left(k+1\right)\left|S\right|}$ such coordinate sequences,
and so, by a union bound:
\[
\Pr\left[\prod_{i\in S}Z_{i}=1\right]\leq m^{\left(k+1\right)\left|S\right|}\frac{\left(n-\left|S\right|\right)!}{n!}\leq\left(\frac{m^{k+1}}{n-\left|S\right|}\right)^{\left|S\right|}
\]
We have: $\left|S\right|\leq\frac{n}{m}=\frac{1}{C}n^{\frac{k}{k+1}}$.
Thus:
\[
\Pr\left[\prod_{i\in S}Z_{i}=1\right]\leq\left(\frac{C^{k+1}n}{n-\frac{1}{C}n^{\frac{k}{k+1}}}\right)^{\left|S\right|}=\left(\left(1+o\left(1\right)\right)C^{k+1}\right)^{\left|S\right|}
\]
and the result follows.
\end{proof}
Lemma \ref{lem:multiplicative prob} allows us to apply theorem \ref{thm:generalized chernoff}
to the variables $\left\{ Z_{i}\right\} _{1\leq i\leq\frac{n}{m}}$
and conclude:
\begin{cor}
\label{cor:chernoff}For all $\beta>C^{k+1}$, for large enough $n$
it holds: 
\[
{\Pr\left[Y_{n}>\beta\right]\leq\exp\left(-n^{\frac{k}{k+1}}\gamma\right)}
\]
for some $\gamma>0$.
\end{cor}
We are now ready to complete the proof of theorem \ref{thm:random monotone}:
\begin{prop}
\label{prop:.lower tail bound}Let $0<C<1$ be the unique solution
to the equation ${C\left(1+C^{k+1}\right)=1}$. Then $\Pr\left[Y_{n}<\frac{1}{k+2}\right]=o\left(1\right)$.\end{prop}
\begin{proof}
By proposition \ref{prop:lower bound expectation}
\begin{equation}
E\left[Y_{n}\right]\geq\frac{C^{k}}{C^{k+1}+1}-o\left(1\right)=C^{k+1}-o\left(1\right)\label{eq:exp lower bound}
\end{equation}

For an integer $n$ and $0<x<C^{k+1}$, let $p_{n}=\Pr\left[Y_{n}\leq x\right]$.
Since $Y_{n}\le\frac{1}{C}+o(1)$ for every $\varepsilon>0$:
\begin{equation}
E\left[Y_{n}\right]\leq p_{n}x+\left(1-p_{n}\right)\left(C^{k+1}+\varepsilon\right)+\left(\frac{1}{C}+o\left(1\right)\right)\Pr\left[Y_{n}\geq C^{k+1}+\varepsilon\right]\label{eq:expectation bound}
\end{equation}
Corollary \ref{cor:chernoff} yields: 
\[
\Pr\left[Y_{n}\geq C^{k+1}+\varepsilon\right]=o\left(1\right)
\]
Combining inequalities \ref{eq:exp lower bound} and \ref{eq:expectation bound}
and rearranging:
\[
p_{n}\left(C^{k+1}-x\right)\leq\varepsilon\left(1-p_{n}\right)+o\left(1\right)
\]
But this holds for all $n$ and $\varepsilon>0$, so that $\lim_{n\rightarrow\infty}p_{n}=0$.

The result follows by taking $x^{k+1}=\frac{1}{k+2}<C^{k+1}$.
\end{proof}

\section{Concluding Remarks and Open Problems}

What are the best possible constant factors in theorems \ref{thm:balanced h.d. est}
and \ref{thm:erdos-szekeres-high-dim}? We have preferred the clarity
of arguments over improved constants, and the bounds we present can
be somewhat improved with some additional effort, but we surely do
not know what the best constants are. Most concretely: What is the
precise statement about the existence of long monotone subsequences
in Latin squares?

For $A\in L_{n}^{k}$ and $\vec{c}\in\left\{ 0,1\right\} ^{k+1}$,
let $\ell_{\vec{c}}\left(A\right)$ be the length of the longest $<_{\vec{c}}$-monotone
subsequence in $A$. Let $\ell\left(A\right)=\left(\ell_{\vec{c}}\left(A\right)\right)_{\vec{c}\in\left\{ 0,1\right\} ^{k+1}}$.
We seek a better description of the set ${\left\{ \ell\left(A\right):A\in L_{n}^{k}\right\} }$.
By theorem \ref{thm:balanced h.d. est} we know that ${\min_{x\in\ell_{n}^{k}}\left\Vert x\right\Vert _{\infty}=\Theta\left(\sqrt{n}\right)}$.
Theorem \ref{thm:erdos-szekeres-high-dim} gives fairly tight sufficient
conditions under which we can conclude that $x_{\vec{c}}\geq r\lor x_{\vec{d}}\geq s$
for $\vec{c},\vec{d}\in\left\{ 0,1\right\} ^{k+1}$ that differ in
precisely one coordinate.

The proof of theorem \ref{thm:random monotone} relied only on very
limited randomness. Recall that $L_{n}^{k}$ splits into \emph{isotopy
classes} where permutations are reachable from each other by applications
of symmetries \ref{enu:permutation of coordinate values} in remark
\ref{rem:Latin symmetries}. Note that theorem \ref{thm:random monotone}
applies when the high-dimensional permutation is chosen uniformly
at random from a particular isotopy class, rather than all of $L_{n}^{k}$.
Beyond the randomness inherent in these symmetries, we have little
insight to offer into the structure of random high-dimensional permutations.
In our view, it's a major challenge in this field to understand (fully)
random high-dimensional permutations. In particular, we do not know
how to uniformly sample elements of $L_{n}^{k}$. Even for Latin squares,
the best known method is Jacobson and Matthews' Markov chain \cite{jacobson1996generating},
which is not known to be rapidly mixing.

We believe theorem \ref{thm:random monotone} can be strengthened,
and there exist constants $c_{k}$ s.t. $H_{n}^{k}n^{-\frac{k}{k+1}}\rightarrow c_{k}$
in probability. This is borne out by numerical experiments, which
indicate that $H_{n}^{2}n^{-\frac{2}{3}}$ is concentrated in a small
interval. We do not know how to prove this, but perhaps super-additive
ergodic theorems \`a la Hammersley \cite{hammersley1972few} may
apply. We also note that the analogous result for random points in
Euclidean space is known \cite{bollobas1988longest}.

\bibliographystyle{plain}
\bibliography{latin_squares}

\end{document}